\documentclass{svmult}
\usepackage{mathptmx}       	
\usepackage{helvet}         	
\usepackage{courier}        	
\usepackage{type1cm} 		
\usepackage[bottom]{footmisc}	
\usepackage{amsmath,amsfonts,url,tikz,float}

\newcommand{\K}{\textbf{\textit{K}}}
\newcommand{\N}{\mathbb N}

\date{}

\begin{document}
\bibliographystyle{spmpsci}
\title*{All functions $g\colon {\mathbb N} \to {\mathbb N}$ which have
a single-fold Diophantine representation are dominated by a limit-computable
function $f\colon {\mathbb N} \setminus \{0\} \to {\mathbb N}$ which is
implemented in \textsl{MuPAD} and whose computability is an open problem}
\titlerunning{All functions $g\colon {\mathbb N} \to {\mathbb N}$ which have
a single-fold Diophantine representation\ldots}
\author{Apoloniusz Tyszka}
\institute{Apoloniusz Tyszka \at University of Agriculture, Faculty
of Production and Power Engineering, Balicka 116B, 30-149 Krak\'ow, Poland,
\email{rttyszka@cyf-kr.edu.pl}
}
\maketitle

\abstract{Let $E_n=\{x_k=1,~x_i+x_j=x_k,~x_i \cdot x_j=x_k\colon i,j,k \in
\{1,\ldots,n\}\}$. For any integer $n \geq 2214$, we define a system
$T \subseteq E_n$ which has a unique integer solution $(a_1,\ldots,a_n)$.
We prove that the numbers $a_1,\ldots,a_n$ are positive and
$\mathrm{max}\left(a_1,\ldots,a_n\right)>2^{\textstyle 2^n}$. For
a positive integer $n$, let $f(n)$ denote the smallest non-negative
integer $b$ such that for each system $S \subseteq E_n$ with a unique
solution in non-negative integers $x_1,\ldots,x_n$, this solution
belongs to $[0,b]^n$. We prove that if a function $g\colon {\mathbb N}
\to {\mathbb N}$ has a single-fold Diophantine representation, then
$f$ dominates $g$. We present a \textsl{MuPAD} code which takes as input
a positive integer $n$, performs an infinite loop, returns a non-negative
integer on each iteration, and returns $f(n)$ on each sufficiently high
iteration.}

\vfill


\vfill


\vspace{3ex}

Let $E_n=\{x_k=1,~x_i+x_j=x_k,~x_i \cdot x_j=x_k\colon i,j,k \in \{1,\ldots,n\}\}$.
The following system
\[
\left\{\begin{array}{rcl}
x_1 			& = & 1\\
x_1+x_1			& = & x_2\\
x_2 \cdot x_2		& = & x_3\\
x_3 \cdot x_3		& = & x_4\\
x_4 \cdot x_4		& = & x_5\\
			& \ldots\\
x_{n-1} \cdot x_{n-1} 	& = & x_n
\end{array}\right.
\]
has a unique complex solution, namely $\left(1,2,4,16,256,\ldots,
2^{\textstyle 2^{n-3}},2^{\textstyle 2^{n-2}}\right)$. The following system
\[
\left\{\begin{array}{rcl}
x_1 + x_1 		& = & x_2\\
x_1 \cdot x_1 		& = & x_2\\
x_2 \cdot x_2 		& = & x_3\\
x_3 \cdot x_3 		& = & x_4\\
			& \ldots\\
x_{n-1} \cdot x_{n-1}	& = & x_n
\end{array}\right.
\]
has exactly two complex solutions, namely:\newline
$(0,\ldots,0)$ and $\left(2,4,16,256,\ldots,2^{\textstyle 2^{n-2}},
2^{\textstyle 2^{n-1}}\right)$.

\begin{theorem}\label{the1}
For each integer $n \geq 2203$, the following system $T$
\[
\left\{\begin{array}{lrcl}
(T_1) & \forall i \in \{1,\ldots,n\} ~x_i \cdot x_i	& = & x_{i+1}\\
(T_2) & x_{n+2} + x_{n+2} 				& = & x_{n+3}\\
(T_3) & x_{n+3} + x_{n+3} 				& = & x_{n+4}\\
(T_4) & x_{n+4} + x_{n+2} 				& = & x_{n+5}\\
(T_5) & x_{n+6} 					& = & 1\\
(T_6) & x_{n+5}+x_{n+6} 				& = & x_{n+7}\\
(T_7) & x_{n+7}+x_{n+6} 				& = & x_{n+8}\\
(T_8) & x_{n+8}+x_{n+6} 				& = & x_1\\
(T_9) & x_{n+8} \cdot x_{n+8} 				& = & x_{n+9}\\
(T_{10}) & x_{n+9} \cdot x_{n+10} 			& = & x_{n+11}\\
(T_{11}) & x_{n+11}+x_1 				& = & x_{2204}
\end{array}\right.
\]
has a unique integer solution $(a_1,\ldots,a_{n+11})$. The numbers $a_1,
\ldots,a_{n+11}$ are positive and $\mathrm{max}\Bigl(a_1,\ldots,
a_{n+11}\Bigr)>2^{\textstyle 2^{n+11}}$.
\end{theorem}

\begin{proof}
Equations $(T_2)$--$(T_7)$ imply that $x_{n+8}=5x_{n+2}+2$. Hence,\newline
$x_{n+8} \not\in \left\{-1,0,1,-2^{2203}+1\right\}$. The system $(T_1)$
implies that $x_{n+1}={x_1}^{\textstyle 2^n}$ and\newline
${x_1}^{\textstyle 2^{2203}}=x_{2204}$. By this and equations $(T_5)$ and
$(T_8)$--$(T_{11})$, we get:
\[
\Bigl(x_{n+8}+1\Bigr)^{\textstyle 2^{2203}}=\Bigl(x_{n+8}+
x_{n+6}\Bigr)^{\textstyle 2^{2203}}={x_1}^{\textstyle 2^{2203}}=
x_{2204}=x_{n+11}+x_1=
\]
\begin{equation}\label{eq1}
\left(x_{n+9} \cdot x_{n+10}\right)+x_1=\left(x_{n+8}^2 \cdot x_{n+10}\right)+
\left(x_{n+8}+x_{n+6}\right)=x_{n+8}^2 \cdot x_{n+10}+x_{n+8}+1
\end{equation}
Next,
\begin{equation}\label{eq2}
\Bigl(x_{n+8}+1\Bigr)^{\textstyle 2^{2203}}=1+2^{2203} \cdot x_{n+8}+x_{n+8}^2
\cdot \sum_{\textstyle k=2}^{\textstyle 2^{2203}} {2^{2203} \choose k}
\cdot x_{n+8}^{k-2}
\end{equation}
\newpage
\noindent
Formulae~(\ref{eq1}) and~(\ref{eq2}) give:
\[
x_{n+8}^2 \cdot \left(x_{n+10} - \sum_{\textstyle k=2}^{\textstyle 2^{2203}}
{2^{2203} \choose k} \cdot x_{n+8}^{k-2}\right)=\Bigl(2^{2203}-1\Bigr)
\cdot x_{n+8}
\]
The number $2^{2203}-1$ is prime (\cite[p. 79 and p. 81]{Ribenboim})
and\newline
$x_{n+8} \not\in \left\{-1,0,1,-2^{2203}+1\right\}$. Hence,
$x_{n+8}=2^{2203}-1$. This proves that exactly one integer tuple
$(x_1,\ldots,x_{n+11})$ solves $T$ and the numbers $x_1,\ldots,x_{n+11}$
are positive. Next, $x_1=x_{n+8}+x_{n+6}=\Bigl(2^{2203}-1\Bigr)+1=2^{2203}$,
and finally
\[
x_{n+1}={x_1}^{\textstyle 2^n}=\left(2^{2203}\right)^{\textstyle 2^n} >
\Bigl(2^{2048}\Bigr)^{\textstyle 2^n}=2^{\textstyle 2^{n+11}}
\]
Explicitly, the whole solution is given by
\[
\left\{\begin{array}{rcl}
\forall i \in \{1,\ldots,n+1\} ~a_i	& = &
		      \left(2^{2203}\right)^{\textstyle 2^{i-1}}\\
a_{n+2} 	& = & \frac{1}{5} \cdot \left(2^{2203}-3\right)\\
a_{n+3} 	& = & \frac{2}{5} \cdot \left(2^{2203}-3\right)\\
a_{n+4} 	& = & \frac{4}{5} \cdot \left(2^{2203}-3\right)\\
a_{n+5} 	& = & 2^{2203}-3\\
a_{n+6} 	& = & 1\\
a_{n+7} 	& = & 2^{2203}-2\\
a_{n+8} 	& = & 2^{2203}-1\\
a_{n+9} 	& = & \left(2^{2203}-1\right)^2\\
a_{n+10} 	& = & 1+\sum_{k=2}^{\textstyle 2^{2203}}\limits
		      {\textstyle 2^{2203} \choose
		      \textstyle k} \cdot \left(2^{2203}-1\right)^{k-2}\\
a_{n+11} 	& = & \left(2^{2203}\right)^{\textstyle 2^{2203}}-2^{2203}
\end{array}\right.
\]
\begin{flushright}
$\qed$
\end{flushright}
\end{proof}

If we replace the equation $(T_5)$ by the system $\forall i \in \{1,\ldots,n+11\}~x_{n+6} \cdot x_i=x_i$,
then the system $T$ contains only equations of the form $x_i+x_j=x_k$ or $x_i \cdot x_j=x_k$,
and exactly two integer tuples solve $T$, namely $(0,\ldots,0)$ and $(a_1,\ldots,a_{n+11})$.
Hence, Theorem~\ref{the1} disproves the conjecture in~\cite{Tyszka1}, where the author proposed
the upper bound~$2^{\textstyle 2^{n-1}}$ for positive integer solutions to any system
\[
S \subseteq \{x_i+x_j=x_k,~x_i \cdot x_j = x_k\colon~i,j,k \in
\{1,\ldots,n\}\}
\]
which has only finitely many solutions in positive integers $x_1,\ldots,x_n$.
Theorem~\ref{the1} disproves the conjecture in~\cite{Tyszka2}, where the
author proposed the upper bound~$2^{\textstyle 2^{n-1}}$ for modulus of
integer solutions to any system $S \subseteq E_n$ which has only finitely many
solutions in integers $x_1,\ldots,x_n$. For each integer $n \geq 2$, the
following system
\[
\left\{\begin{array}{rcl}
\forall i \in \{1,\ldots,n\} ~x_i \cdot x_i	& = & x_{i+1}\\
x_{n+2} 					& = & 1\\
x_{n+2}+x_{n+2} 				& = & x_{n+3}\\
x_{n+3}+x_{n+3} 				& = & x_{n+4}\\
x_{n+4}+x_{n+5} 				& = & x_{n+6}\\
x_{n+6}+x_{n+2} 				& = & x_1\\
x_{n+6} \cdot x_{n+6} 				& = & x_{n+7}\\
x_{n+8}+x_{n+8} 				& = & x_{n+9}\\
x_{n+9}+x_{n+2} 				& = & x_{n+10}\\
x_{n+7} \cdot x_{n+10} 				& = & x_{n+11}\\
x_{n+11}+x_{n+2} 				& = & x_{n+1}
\end{array}\right.
\]
has a unique solution $(a_1,\ldots,a_{n+11})$ in non-negative integers
(\cite{IPL}). The proof of this gives also that $a_{n+1}>2^{\textstyle
2^{(n+11)-2}}$ for any $n \geq 512$ (\cite{IPL}). The above-described result
inspired the author to formulate Theorem~\ref{the1} and the next
Theorem~\ref{the2}.

\begin{theorem}\label{the2}
If $n \in {\mathbb N}$ and $2^n-1$ is prime, then the following system
\[
\left\{\begin{array}{rcl}
\forall i \in \{1,\ldots,n\} ~x_i \cdot x_i &=& x_{i+1} \\
x_{n+2} &=& 1 \\
x_{n+3}+x_{n+2} &=& x_{n+4} \\
x_{n+4}+x_{n+2} &=& x_{n+5} \\
x_{n+5}+x_{n+2} &=& x_1 \\
x_{n+5} \cdot x_{n+5} &=& x_{n+6} \\
x_{n+6} \cdot x_{n+7} &=& x_{n+8} \\
x_{n+8}+x_1 &=& x_{n+1}
\end{array}\right.
\]
has a unique solution $\left(x_1,\ldots,x_{n+8}\right)$ in non-negative
integers and ${\rm max}\left(x_1,\ldots,x_{n+8}\right)=x_{n+1}=\left(2^n\right)^{\textstyle 2^n}$.
\end{theorem}
\begin{proof}
The proof is analogous to that of Theorem~\ref{the1}.
\begin{flushright}
$\qed$
\end{flushright}
\end{proof}
\begin{theorem}\label{the3}
If $n \in {\mathbb N}\setminus\{0\}$ and $2^{\textstyle 2^n}+1$ is prime,
then the following system
\[
\left\{\begin{array}{rcl}
\forall i \in \{1,\ldots,n\} ~x_i \cdot x_i &=& x_{i+1} \\
x_{n+2} &=& 1 \\
x_{1}+x_{n+2} &=& x_{n+3} \\
x_{n+3}+x_{n+2} &=& x_{n+4} \\
x_{n+1}+x_{n+2} &=& x_{n+5} \\
x_{n+4} \cdot x_{n+6} &=& x_{n+5}
\end{array}\right.
\]
has a unique solution $\left(a_1,\ldots,a_{n+6}\right)$ in non-negative
integers. The numbers $a_1,\ldots,a_{n+6}$ are positive and
${\rm max}\left(a_1,\ldots,a_{n+6}\right)=a_{n+5}=\left(2^{\textstyle 2^n}-1\right)^{\textstyle 2^n}+1$.
\end{theorem}
\begin{proof}
The system equivalently expresses that $(x_1+2) \cdot x_{n+6}=x_1^{\textstyle 2^n}+1$.
Since
\[
x_1^{\textstyle 2^n}+1=1+\left((x_1+2)-2\right)^{\textstyle 2^n}=
\]
\[
1+2^{\textstyle 2^n}+(x_1+2) \cdot \sum_{\textstyle k=1}^{\textstyle 2^n}
{\textstyle 2^n \choose k} \cdot (x_1+2)^{\textstyle k-1} \cdot (-2)^{\textstyle 2^n-k}
\]
we get
\[
(x_1+2) \cdot x_{n+6}=1+2^{\textstyle 2^n}+(x_1+2) \cdot
\sum_{\textstyle k=1}^{\textstyle 2^n} {{\textstyle 2^n} \choose {\textstyle k}}
\cdot (x_1+2)^{\textstyle k-1} \cdot (-2)^{\textstyle 2^n-k}
\]
Hence, $x_1+2$ divides $1+2^{\textstyle 2^n}$.
Since $x_1+2 \geq 2$ and $1+2^{\textstyle 2^n}$ is prime, we get
$x_1+2=1+2^{\textstyle 2^n}$ and $x_1=2^{\textstyle 2^n}-1$. Next,
$x_{n+1}=x_1^{\textstyle 2^n}=\left(2^{\textstyle 2^n}-1\right)^{\textstyle 2^n}$ and
\[
x_{n+5}=x_{n+1}+x_{n+2}=\left(2^{\textstyle 2^n}-1\right)^{\textstyle 2^n}+1
\]
Explicitly, the whole solution is given by
\[
\left\{\begin{array}{rcl}
\forall i \in \{1,\ldots,n+1\} ~a_i &=& \left(2^{\textstyle 2^n}-1\right)^{\textstyle 2^{i-1}} \\
a_{n+2} &=& 1 \\
a_{n+3} &=& 2^{\textstyle 2^n} \\
a_{n+4} &=& 2^{\textstyle 2^n}+1 \\
a_{n+5} &=& \left(2^{\textstyle 2^n}-1\right)^{\textstyle 2^n}+1 \\
a_{n+6} &=& 1+\sum_{\textstyle k=1}^{\textstyle 2^{n}}\limits
\displaystyle {{\textstyle 2^n} \choose {\textstyle k}} \cdot \left(2^{\textstyle 2^n}+1\right)^{\textstyle k-1}
\cdot (-2)^{\textstyle 2^n-k}
\end{array}\right.
\]
\begin{flushright}
$\qed$
\end{flushright}
\end{proof}
\par
It is conjectured that $2^{\textstyle 2^n}+1$ is prime only for $n \in \{0,1,2,3,4\}$,
although it is still not excluded that $2^{\textstyle 2^n}+1$ is prime for each sufficiently
large positive integer $n$. Unconditionally, for each positive integer $n$, the following system
\[
\left\{\begin{array}{rcl}
\forall i \in \{1,\ldots,n\} ~x_i \cdot x_i &=& x_{i+1} \\
x_{n+2} &=& 1 \\
x_{n+3}+x_{n+2} &=& x_1 \\
x_{n+4}+x_{n+2} &=& x_{n+3} \\
x_{n+4} \cdot x_{n+5} & = & x_{n+1}
\end{array}\right.
\]
has only finitely many integer solutions $\left(x_1,\ldots,x_{n+5}\right)$.
The maximal solution is given by 
\begin{displaymath}
\left\{\begin{array}{rcl}
\forall i \in \{1,\ldots,n+1\} ~x_i &=& \left(2+2^{\textstyle 2^n}\right)^{\textstyle 2^{i-1}} \\
x_{n+2} &=& 1 \\
x_{n+3} &=& 1+2^{\textstyle 2^n} \\
x_{n+4} &=& 2^{\textstyle 2^n} \\
x_{n+5} &=& \left(1+2^{\textstyle 2^n-1}\right)^{\textstyle 2^n}
\end{array}\right.
\end{displaymath}
\par
The Davis-Putnam-Robinson-Matiyasevich theorem states that every recursively
enumerable set {${\cal M} \subseteq {{\mathbb N}}^n$} has a Diophantine
representation, that is
\begin{equation}
(a_1,\ldots,a_n) \in {\cal M} \Longleftrightarrow \exists x_1, \ldots, x_m \in
{\mathbb N} ~~W(a_1,\ldots,a_n,x_1,\ldots,x_m)=0 \tag*{\texttt{(R)}}
\end{equation}
for some polynomial $W$ with integer coefficients, see \cite{Matiyasevich1}.
The polynomial~$W$ can be computed, if we know the Turing machine~$M$ such
that, for all\newline
{$(a_1,\ldots,a_n) \in {{\mathbb N}}^n$}, $M$ halts on {$(a_1,\ldots,a_n)$} if
and only if {$(a_1,\ldots,a_n) \in {\cal M}$}, see \cite{Matiyasevich1}. The
representation~\texttt{(R)} is said to be {single-fold}, if for any
{$a_1,\ldots,a_n \in {\mathbb N}$} the equation
{$W(a_1,\ldots,a_n,x_1,\ldots,x_m)=0$} has at most one solution
{$(x_1,\ldots,x_m) \in {{\mathbb N}}^m$}. Y. Matiyasevich conjectures that
each recursively enumerable set {${\cal M} \subseteq {{\mathbb N}}^n$} has a
{single-fold} Diophantine representation, see {\cite[pp.~341--342]{DMR}},
{\cite[p.~42]{Matiyasevich2}}, {\cite[p.~79]{Matiyasevich3a}}, and \mbox{\cite[p.~745]{Matiyasevich3b}}.\\[1ex]

Let us say that a set {${\cal M} \subseteq {\mathbb N}^n$} has a bounded Diophantine
representation, if there exists a polynomial $W$ with integer coefficients such that
\[
(a_1,\ldots,a_n) \in {\cal M} \Longleftrightarrow
\]
\[
\exists x_1,\ldots,x_m \in \left\{0,\ldots,\mathrm{max}\left(a_1,\ldots,a_n\right)\right\}
~W\left(a_1,\ldots,a_n,x_1,\ldots,x_m\right)=0
\]
Of course, any bounded Diophantine representation is finite-fold
and any subset of ${\mathbb N}$ with a bounded Diophantine
representation is computable. A simple diagonal argument shows that there exists
a computable subset of ${\mathbb N}$ without any bounded Diophantine
representation, see \cite[p.~360]{DMR}. The authors of \cite{DMR}
suggest a possibility that each subset of ${\mathbb N}$ which
has a finite-fold Diophantine representation has also a bounded
Diophantine representation, see \cite[p.~360]{DMR}.

Let $\omega$ denote the least infinite cardinal number, and let $\omega_1$
denote the least uncountable cardinal number. Let
{$\kappa \in \left\{2,3,4,\ldots,\omega,\omega_1\right\}$}. We say that the
representation~$\texttt{(R)}$ is {$\kappa$-fold}, if for any
{$a_1,\ldots,a_n \in {\mathbb N}$} the equation
\mbox{$W\left(a_1,\ldots,a_n,x_1,\ldots,x_m\right)=0$} has less than $\kappa$
solutions {$\left(x_1,\ldots,x_m\right) \in {{\mathbb N}}^m$}.
Of course, $2$-fold Diophantine representations are identical to single-fold
Diophantine representations, $\omega$-fold Diophantine representations
are identical to finite-fold Diophantine representations, and $\omega_1$-fold
Diophantine representations are identical to Diophantine representations.
\newpage
For a positive integer $n$, let {$f_{\textstyle \kappa}(n)$} denote the smallest non-negative
integer $b$ such that for each system {$S \subseteq E_n$} which has a solution
in non-negative integers {$x_1,\ldots,x_n$} and which has less than~$\kappa$
solutions in non-negative integers {$x_1,\ldots,x_n$}, there exists a
solution of~$S$ in non-negative integers not greater than~$b$. For a
positive integer $n$, let {$f(n)$} denote the smallest non-negative integer
$b$ such that for each system {$S \subseteq E_n$} with a unique solution in
non-negative integers {$x_1,\ldots,x_n$}, this solution belongs to
{$[0,b]^n$}. Obviously, {$f=f_2$}, {$f(1)=1$}, and {$f(2)=2$}.
\begin{lemma}\label{lem2}
(\cite{Lagarias}) If {$k \in {\mathbb N}$}, then the equation
{$x^2+1=5^{2k+1} \cdot y^2$} has infinitely many solutions in non-negative
integers. The minimal solution is given by
\[
x=\frac{\left(2+\sqrt{5}\right)^{\textstyle 5^k} +
\left(2-\sqrt{5}\right)^{\textstyle5^{k}}}{2}
\]
\[
y=\frac{\left(2+\sqrt{5}\right)^{\textstyle 5^k} -
\left(2-\sqrt{5}\right)^{\textstyle5^{k}}}{2 \cdot \sqrt{5} \cdot 5^k}
\]
\end{lemma}
\begin{theorem}\label{the4}
For each positive integer $n$, the following system
\[
\left\{\begin{array}{rcl}
\forall i \in \{1,\ldots,n\} ~x_i \cdot x_i	& = & x_{i+1}\\
x_1 \cdot x_{n+1} 				& = & x_{n+2}\\
x_{n+3} 					& = & 1\\
x_{n+3}+x_{n+3} 				& = & x_{n+4}\\
x_{n+4}+x_{n+4} 				& = & x_{n+5}\\
x_{n+5}+x_{n+3}					& = & x_1\\
x_{n+6} \cdot x_{n+6} 				& = & x_{n+7}\\
x_{n+8} \cdot x_{n+8} 				& = & x_{n+9}\\
x_{n+9}+x_{n+3} 				& = & x_{n+10}\\
x_{n+2} \cdot x_{n+7} 				& = & x_{n+10}
\end{array}\right.
\]
has infinitely many solutions in non-negative integers $x_1,\ldots,x_{n+10}$.
If an integer tuple {$(x_1,\ldots,x_{n+10})$} solves the system, then
\[
x_{n+10} \geq \left(\frac{\textstyle
\left(2+\sqrt{5}\right)^{\textstyle 5^{2^{\textstyle n-1}}}+
\textstyle \left(2-\sqrt{5}\right)^{\textstyle 5^{\textstyle 2^{n-1}}}}
{\textstyle 2}\right)^2+1
\]
\end{theorem}

\begin{proof}
If follows from Lemma~\ref{lem2}, because the system equivalently expresses
that $x_{n+10}=x_8^2+1=5^{\textstyle 2 \cdot 2^{n-1}+1} \cdot x_{n+6}^2+1$.
\begin{flushright}
$\qed$
\end{flushright}
\end{proof}
\newpage
Let {\textsl{Rng}} denote the class of all rings $\K$ that extend ${\mathbb Z}$.
\begin{lemma}\label{lem3}
(\cite{Tyszka1}) Let {$D(x_1,\ldots,x_p) \in {\mathbb Z}[x_1,\ldots,x_p]$}.
Assume that {$\mathrm{deg}(D,x_i) \geq 1$} for each {$i \in \{1,\ldots,p\}$}.
We can compute a positive integer {$n>p$} and a system {$T \subseteq E_n$}
which satisfies the following two conditions:\\[2ex]
\verb+Condition 1.+ If {$\K \in \textsl{Rng} \cup
\{{\mathbb N},~{\mathbb N} \setminus \{0\}\}$}, then
\[
\forall \tilde{x}_1,\ldots,\tilde{x}_p \in
\K ~\Bigl(D(\tilde{x}_1,\ldots,\tilde{x}_p)=0 \Longleftrightarrow
\]
\[
\exists \tilde{x}_{p+1},\ldots,\tilde{x}_n \in \K ~(\tilde{x}_1,\ldots,
\tilde{x}_p,\tilde{x}_{p+1},\ldots,\tilde{x}_n) ~solves~ T\Bigr)
\]
\verb+Condition 2.+ If {$\K \in \textsl{Rng} \cup
\{{\mathbb N},~{\mathbb N} \setminus \{0\}\}$}, then for each
{$\tilde{x}_1,\ldots,\tilde{x}_p \in \K$} with\newline
{$D(\tilde{x}_1,\ldots,\tilde{x}_p)=0$}, there exists a unique tuple
{$(\tilde{x}_{p+1},\ldots,\tilde{x}_n) \in {\K}^{n-p}$} such that the tuple
{$(\tilde{x}_1,\ldots,\tilde{x}_p,\tilde{x}_{p+1},\ldots,\tilde{x}_n)$}
{solves $T$}.\\[2ex]
Conditions 1 and 2 imply that for each {$\K \in \textsl{Rng} \cup
\{{\mathbb N},~{\mathbb N} \setminus \{0\}\}$}, the equation
{$D(x_1,\ldots,x_p)=0$} and the {system $T$} have the same number of solutions
{in $\K$}.
\end{lemma}
\par
Theorems~\ref{the2} and \ref{the4} provide a heuristic argument that the
function {$f_{\textstyle \omega_1}$} grows much faster than the function~$f$.
The next Theorem~\ref{the5} for $\kappa=\omega_1$ implies that the function~$f_{\textstyle \omega_1}$
is not computable. These facts lead to the conjecture that the function~$f$ is computable.
By this, Theorem~\ref{the5} for $\kappa=2$ is the first step towards disproving
Matiyasevich's conjecture on single-fold Diophantine representations.
\begin{theorem}\label{the5}
If a function {$g\colon {\mathbb N} \to {\mathbb N}$} has a {$\kappa$-fold}
Diophantine representation, then there exists a positive integer~$m$ such that
{$g(n)<f_{\textstyle \kappa}(n)$} for any {$n \geq m$}.
\end{theorem}
\begin{proof}
By Lemma~\ref{lem3} for {$\K={\mathbb N}$}, there is an integer {$s \geq 3$}
such that for any non-negative integers~{$x_1,x_2$},
\begin{equation}
(x_1,x_2) \in g \Longleftrightarrow \exists x_3,\ldots,x_s \in {\mathbb N}\ \ 
\Phi(x_1,x_2,x_3,\ldots,x_s),\tag*{\texttt{(E)}}
\end{equation}
where the formula $\Phi(x_1,x_2,x_3,\ldots,x_s)$ is a conjunction of
formulae of the forms \mbox{$x_k=1$}, $x_i+x_j=x_k$, $x_i \cdot x_j=x_k$
$(i,j,k \in \{1,\ldots,s\})$, and for each non-negative integers $x_1,x_2$
less than $\kappa$ tuples {$(x_3,\ldots,x_s) \in {{\mathbb N}}^{s-2}$} satisfy
{$\Phi(x_1,x_2,x_3,\ldots,x_s)$}. Let $[\cdot]$ denote the integer part
function. For each integer {$n \geq 6+2s$},
\[
n-\left[\frac{n}{2}\right]-3-s \geq 6+2s-\left[\frac{6+2s}{2}\right]-3-s \geq
6+2s-\frac{6+2s}{2}-3-s=0
\]
For an integer {$n \geq 6+2s$}, let $S_n$ denote the following system
\[
\left\{
\begin{array}{c}
\textrm{all~equations~occurring~in~}\Phi(x_1,x_2,x_3,\ldots,x_s)\\
n-\left[\frac{n}{2}\right]-3-s \textrm{~~equations~of~the~form~} z_i = 1\\
\begin{array}{rcl}
t_1				& = & 1\\
t_1+t_1 			& = & t_2\\
t_2+t_1 			& = & t_3\\
				& \ldots\\
t_{\left[\frac{n}{2}\right]-1}+t_1	& = & t_{\left[\frac{n}{2}\right]}\\
t_{\left[\frac{n}{2}\right]}+t_{\left[\frac{n}{2}\right]}	& = & w\\
w+y				& = & x_1\\
y+y				& = & y \textrm{~(if~}n\textrm{~is~even)}\\
y				& = & 1 \textrm{~(if~}n\textrm{~is~odd)}\\
x_2+t_1				& = & u
\end{array}
\end{array}
\right.
\]
with $n$ variables. The system $S_n$ has less than $\kappa$ solutions
in~${{\mathbb N}}^n$. By the equivalence~\texttt{(E)}, $S_n$ is satisfiable
over~${\mathbb N}$. If a {$n$-tuple} {$(x_1,x_2,x_3,\ldots,x_s,\ldots,w,y,u)$}
of non-negative
integers solves $S_n$, then by the equivalence~\texttt{(E)},
\[
x_2=g(x_1)=g(w+y)=g\left(2 \cdot \left[\frac{n}{2}\right]+y\right)=g(n)
\]
Therefore, {$u=x_2+t_1=g(n)+1>g(n)$}. This shows that
{$g(n)<f_{\textstyle \kappa}(n)$} for any {$n \geq 6+2s$}.
\begin{flushright}
$\qed$
\end{flushright}
\end{proof}
\par
For \mbox{$\kappa \in \{2,3,4,\ldots,\omega,\omega_1\}$} and \mbox{$n \in \N \setminus \{0\}$},
let $B(\kappa,n)$ denote the set of all polynomials \mbox{$D(x,x_1,\ldots,x_i)$}
with integer coefficients that satisfy the following conditions:
\vskip 0.1truecm
the degree of \mbox{$D(x,x_1,\ldots,x_i)$} is not greater than $n$,
\vskip 0.1truecm
each coefficient of \mbox{$D(x,x_1,\ldots,x_i)$} belongs to $[-n,n]$,
\vskip 0.1truecm
$i \leq n$,
\vskip 0.1truecm
for each \mbox{non-negative} integer $j$, the equation \mbox{$D(j,x_1,\ldots,x_i)=0$}
is soluble in \mbox{non-negative} integers \mbox{$x_1,\ldots,x_i$},
\vskip 0.1truecm
for each \mbox{non-negative} integer $j$, the equation \mbox{$D(j,x_1,\ldots,x_i)=0$}
has less than $\kappa$ solutions in \mbox{non-negative} integers \mbox{$x_1,\ldots,x_i$}.
\vskip 0.3truecm
\par
For \mbox{$\kappa \in \{2,3,4,\ldots,\omega,\omega_1\}$} and $n \in \N \setminus \{0\}$,
let $h_{\textstyle \kappa}(n)$ denote the smallest
non-negative integer~$b$ such that for each polynomial
\mbox{$D(x,x_1,\ldots,x_i) \in B(\kappa,n)$} the equation
\mbox{$D(n,x_1,\ldots,x_i)=0$} has a solution in non-negative integers
\mbox{$x_1,\ldots,x_i$} not greater than $b$.
\newpage
\begin{theorem}\label{new}
If a function {$g\colon \N \to \N$} has a \mbox{$\kappa$-fold} Diophantine representation,
then \mbox{$g(n)<h_{\textstyle \kappa}(n)$} for any sufficiently large integer $n$.
\end{theorem}
\begin{proof}
There exists a polynomial~$W(x,x_1,x_2,\ldots,x_i)$ with integer coefficients such that for
any \mbox{non-negative} integers $x,x_1$,
\[
(x,x_1) \in g \Longleftrightarrow \exists x_2,\ldots,x_i \in \N ~~W(x,x_1,x_2,\ldots,x_i)=0
\]
and the above equivalence defines a \mbox{$\kappa$-fold} Diophantine representation of the function $g$.
Hence, for each \mbox{non-negative} integer $j$, the equation
\[
W^2\left(j,x_1,x_2,\ldots,x_i\right)+\left(x_{i+1}-x_1-1\right)^2=0
\]
is soluble in \mbox{non-negative} integers \mbox{$x_1,x_2,\ldots,x_i,x_{i+1}$} and
the number of solutions is smaller than $\kappa$. Let $d$ denote the degree of the polynomial
\begin{equation}\label{eq3}
W^2\left(x,x_1,x_2,\ldots,x_i\right)+\left(x_{i+1}-x_1-1\right)^2
\end{equation}
and let $m$ denote the maximum of the modulus of the coefficients.
If \mbox{$n \geq {\rm max}(d,m,i+1)$}, then polynomial~(\ref{eq3}) belongs to $B(\kappa,n)$
and each solution \mbox{$\left(x_1,\ldots,x_{i+1}\right) \in {\N}^{i+1}$} to
\[
W^2\left(n,x_1,x_2,\ldots,x_i\right)+\left(x_{i+1}-x_1-1\right)^2=0
\]
satisfies \mbox{$x_1=g(n)$} and \mbox{$x_{i+1}=g(n)+1$}. Therefore,
\[
g(n)<g(n)+1 \leq h_{\textstyle \kappa}(n)
\]
for any integer \mbox{$n \geq {\rm max}(d,m,i+1)$}.
\begin{flushright}
$\qed$
\end{flushright}
\end{proof}
\par
Theorem~\ref{new} and the Davis-Putnam-Robinson-Matiyasevich theorem imply that the function
$h_{\textstyle \omega_1}$ dominates all computable functions.
If \mbox{$\kappa \neq \omega_1$}, then the possibility that the function $h_{\textstyle \kappa}$
is majorized by a computable function is still not excluded.
\vskip 0.2truecm
\par
Let us fix an integer {$\kappa \geq 2$}.

For a positive integer $n$, let $\theta(n)$ denote the smallest non-negative
integer $b$ such that for each system $S \subseteq E_n$ with more than
$\kappa-1$ solutions in non-negative integers $x_1,\ldots,x_n$,
at least two such solutions belong to $[0,b]^n$.

For a positive integer $n$ and for a non-negative integer $m$, let
$\beta(n,m)$ denote the smallest non-negative integer~$b$ such that for
each system $S \subseteq E_n$ which has a solution in integers
$x_1,\ldots,x_n$ from the range of $0$ to $m$ and which has less than
$\kappa$ solutions in integers $x_1,\ldots,x_n$ from the range of $0$ to $m$,
there exists a solution that belongs to $[0,b]^n$. The function
$\beta: \left({\mathbb N} \setminus \{0\}\right) \times
{\mathbb N} \to {\mathbb N}$ is computable.
\newpage
The following equalities
\[
f_{\textstyle \kappa}(n)=\beta(n,\textrm{max}(f_{\textstyle \kappa}(n),
\theta(n)))=\beta(n,\mathrm{max}(f_{\textstyle \kappa}(n),\theta(n))+1)=
\]
\[
\beta(n,\mathrm{max}(f_{\textstyle \kappa}(n),\theta(n))+2)=\beta(n,
\mathrm{max}(f_{\textstyle \kappa}(n),\theta(n))+3)=\ldots
\]
hold for any positive integer $n$. Therefore, there is an algorithm which
takes as input a positive integer $n$, performs an infinite loop, returns
$\beta(n,m-1)$ on the $m$-th iteration, and returns $f_{\textstyle \kappa}(n)$
on each sufficiently high iteration. This proves that the function
$f_{\textstyle \kappa}$ is computable in the limit for any integer
$\kappa \geq 2$.

\begin{theorem}\label{the6}
Let {$\kappa=2$}. We claim that the following \textsl{MuPAD} code implements
an algorithm which takes as input a positive integer $n$, performs an infinite
loop, returns $\beta(n,m-1)$ on the {$m$-th} iteration, and returns $f(n)$ on
each sufficiently high iteration.
\begin{quote}
\begin{verbatim}
input("input the value of n",n):
X:=[0]:
while TRUE do
Y:=combinat::cartesianProduct(X $i=1..n):
W:=combinat::cartesianProduct(X $i=1..n):
for s from 1 to nops(Y) do
for t from 1 to nops(Y) do
m:=0:
for i from 1 to n do
if Y[s][i]=1 and Y[t][i]<>1 then m:=1 end_if:
for j from i to n do
for k from 1 to n do
if Y[s][i]+Y[s][j]=Y[s][k] and Y[t][i]+Y[t][j]<>Y[t][k]
then m:=1 end_if:
if Y[s][i]*Y[s][j]=Y[s][k] and Y[t][i]*Y[t][j]<>Y[t][k]
then m:=1 end_if:
end_for:
end_for:
end_for:
if m=0 and s<>t then
W:=listlib::setDifference(W,[Y[s]]) end_if:
end_for:
end_for:
print(max(max(W[z][u] $u=1..n) $z=1..nops(W))):
X:=append(X,nops(X)):
end_while:
\end{verbatim}
\end{quote}
\end{theorem}
\begin{proof}
Let us say that a tuple $y=(y_1,\ldots,y_n) \in {{\mathbb N}}^n$ is a
\emph{duplicate} of a tuple\newline
{$x=(x_1,\ldots,x_n) \in {{\mathbb N}}^n$}, if
\[
(\forall i \in \{1,\ldots,n\} ~(x_i=1 \Longrightarrow y_i=1)) ~\wedge
\]
\[
(\forall i,j,k \in \{1,\ldots,n\} ~(x_i+x_j=x_k \Longrightarrow y_i+y_j=y_k))
\ \wedge
\]
\[
(\forall i,j,k \in \{1,\ldots,n\} ~(x_i \cdot x_j=x_k \Longrightarrow y_i
\cdot y_j=y_k))
\]
For a positive integer $n$ and for a non-negative integer $m$, $\beta(n,m)$
equals the smallest non-negative integer $b$ such that the box $[0,b]^n$
contains all tuples\newline
$(x_1,\ldots,x_n) \in \{0,\ldots,m\}^n$ which have no duplicates in
$\{0,\ldots,m\}^n \setminus \{(x_1,\ldots,x_n)\}$.
\begin{flushright}
$\qed$
\end{flushright}
\end{proof}

The proof of Theorem~\ref{the6} effectively shows that the function $f$ is
computable in the limit. Limit-computable functions, also known as
trial-and-error computable functions, have been thoroughly studied, see
\cite[pp.~233--235]{Soare} for the main results. The function
$f_{\textstyle \omega_1}$ is also computable in the limit (\cite{Tyszka4})
and the following \textsl{MuPAD} code
\begin{quote}
\begin{verbatim}
input("input the value of n",n):
X:=[0]:
while TRUE do
Y:=combinat::cartesianProduct(X $i=1..n):
W:=combinat::cartesianProduct(X $i=1..n):
for s from 1 to nops(Y) do
for t from 1 to nops(Y) do
m:=0:
for i from 1 to n do
if Y[s][i]=1 and Y[t][i]<>1 then m:=1 end_if:
for j from i to n do
for k from 1 to n do
if Y[s][i]+Y[s][j]=Y[s][k] and Y[t][i]+Y[t][j]<>Y[t][k]
then m:=1 end_if:
if Y[s][i]*Y[s][j]=Y[s][k] and Y[t][i]*Y[t][j]<>Y[t][k]
then m:=1 end_if:
end_for:
end_for:
end_for:
if m=0 and max(Y[t][i] $i=1..n)<max(Y[s][i] $i=1..n)
then W:=listlib::setDifference(W,[Y[s]]) end_if:
end_for:
end_for:
print(max(max(W[z][u] $u=1..n) $z=1..nops(W))):
X:=append(X,nops(X)):
end_while:
\end{verbatim}
\end{quote}
performs an infinite computation of $f_{\textstyle \omega_1}(n)$. Flowchart 1
describes an algorithm which computes $f_{\textstyle \kappa}(n)$ in the limit
for any {$\kappa \in \{\omega_1\} \cup \{2,3,4,\ldots\}$}.
\begin{figure}[H]
\centering
\begin{tikzpicture}[thick]
\ttfamily
\path (0,7.4) node {Start};
\path (0,6.6) node {Set $\kappa=\omega_1$ or input an integer
      $\kappa \geq 2$};
\path (0,5.8) node {Input a positive integer $n$};
\path (0,5) node {$m := 0$};
\path (0,3.8) node {Create a list $\cal{L}$ of all systems
      $S \subseteq E_{n}$};
\path (0,3.4) node {which have a solution in $\{0,\ldots,m\}^{n}$};
\path (0,2.6) node {If $\kappa \neq \omega_1$, then remove from
      ${\cal L}$ all systems which};
\path (0,2.2) node {have more than $\kappa - 1$ solutions in
      $\{0,\ldots,m\}^n$};
\path (0,1.4) node {Print the smallest non-negative};
\path (0,1) node {integer $b$ such that each element};
\path (0,.6) node {of $\cal{L}$ has a solution in $\{0,\ldots,b\}^{n}$};
\path (0,-.2) node {$m := m + 1$};
\draw (0,7.4) ellipse (.75cm and 0.2cm);
\draw (-3.5,6.4) -- (-3,6.8) -- (3.5,6.8) -- (3,6.4) -- (-3.5,6.4);
\draw (-2.8,5.6) -- (-2.3,6) -- (2.8,6) -- (2.3,5.6) -- (-2.8,5.6);
\draw (-.55,4.8) rectangle (.55,5.2);
\draw (-3.4,3.2) rectangle (3.4,4);
\draw (-4.3,2) rectangle (4.3,2.8);
\draw (-4.3,.4) -- (-2.8,1.6) -- (4.3,1.6) -- (2.8,.4) -- (-4.3,.4);
\draw (-.85,0) rectangle (.85,-.4);
\draw (.85,-.2) -- (5.25,-.2) -- (5.25,4.4) -- (5.15,4.4);
\draw[->] (5.25,4.4) -- (0,4.4);
\draw[->] (0,7.2) -- (0,6.8);
\draw[->] (0,6.4) -- (0,6);
\draw[->] (0,5.6) -- (0,5.2);
\draw[->] (0,4.8) -- (0,4);
\draw[->] (0,3.2) -- (0,2.8);
\draw[->] (0,2) -- (0,1.6);
\draw[->] (0,.4) -- (0,0);
\end{tikzpicture}\\[2ex]
\textbf{Flowchart~1:} An infinite computation of $f_{\textstyle \kappa}(n)$,
where $\kappa \in \{\omega_1\} \cup \{2,3,4,\ldots\}$
\end{figure}

\textsl{MuPAD} is a computer algebra system whose syntax is modelled on
\textsl{Pascal}. The commercial version of \textsl{MuPAD} is no longer
available as a {stand-alone} product, but only as the \textsl{Symbolic Math
Toolbox} of \textsl{MATLAB}. Fortunately, all presented codes can be executed
by \textsl{MuPAD Light}, which was and is free, see \cite{Tyszka3}.
\begin{theorem}\label{the7}
(\cite{Tyszka4}) Let {$\kappa \in \{2,3,4,\ldots,\omega\}$}. Let us consider
the following three statements:\\[1ex]
$(a)$ There exists an algorithm~${\cal A}$ whose execution always terminates
and which takes as input a Diophantine
equation~$D$ and returns the answer \texttt{YES} or \texttt{NO} which
indicates whether or not the equation~$D$ has a solution in non-negative
integers, if the solution set $Sol(D)$ satisfies
$\textrm{card}(Sol(D))<\kappa$.\\[1ex]
$(b)$ The function $f_{\textstyle \kappa}$ is majorized by a computable
function.\\[1ex]
$(c)$ If a set ${\cal M} \subseteq {{\mathbb N}}^n$ has a $\kappa$-fold
Diophantine representation, then ${\cal M}$ is computable.\\[1ex]
We claim that $(a)$ is equivalent to $(b)$ and $(a)$ implies $(c)$.
\end{theorem}
\begin{proof}
The implication {$(a) \Rightarrow (c)$} is obvious. We prove the implication
$(a) \Rightarrow (b)$. There is an algorithm $\textrm{Dioph}$ which takes as
input a positive integer $m$ and a non-empty system {$S \subseteq E_m$}, and
returns a Diophantine equation {$\textrm{Dioph}(m,S)$} which has the same
solutions in non-negative integers {$x_1,\ldots,x_m$}. Item~$(a)$ implies that
for each Diophantine equation $D$, if the algorithm~${\cal A}$ returns
\texttt{YES} for $D$, then $D$ has a solution in non-negative integers. Hence,
if the algorithm~${\cal A}$ returns \texttt{YES} for {$\mathrm{Dioph}(m,S)$},
then we can compute the smallest non-negative integer $i(m,S)$ such that
$\textrm{Dioph}(m,S)$ has a solution in non-negative integers not greater than
$i(m,S)$. If the algorithm~${\cal A}$ returns \texttt{NO} for
$\mathrm{Dioph}(m,S)$, then we set {$i(m,S)=0$}. The function
\[
{\mathbb N} \setminus \{0\} \ni m \to \textrm{max}\Bigl\{i(m,S)\colon\emptyset
\neq S \subseteq E_m\Bigr\} \in {\mathbb N}
\]
is computable and majorizes the function $f_{\textstyle \kappa}$. We prove the
implication $(b) \Rightarrow (a)$.
Let a function $h$ majorizes~$f_{\textstyle \kappa}$. By Lemma~\ref{lem3} for
$\K={\mathbb N}$, a Diophantine equation~$D$ is equivalent to a system
$S \subseteq E_n$. The algorithm~${\cal A}$ checks whether or not $S$ has a
solution in non-negative integers {$x_1,\ldots,x_n$} not greater than $h(n)$.
\begin{flushright}
$\qed$
\end{flushright}
\end{proof}

The implication {$(a) \Rightarrow (c)$} remains true with a weak formulation
of item~$(a)$, where the execution of~${\cal A}$ may not terminate or
${\cal A}$ may return nothing or something irrelevant, if $D$ has at least
$\kappa$ solutions in non-negative integers. The weakened item~$(a)$ implies
that Flowchart~2

\begin{figure}[H]
\centering
\begin{tikzpicture}[thick]
\ttfamily
\path (0,6.2) node{Start};
\path (0,5.4) node{Input a Diophantine equation $D (x_{1},\ldots,x_{n}) = 0$};
\path (0,4.6) node{$m := 0$};
\path (0,3.4) node{Execute $\mathcal{A}$ on the equation};
\path (0,3) node{$(m + y - (x_{1} + \ldots + x_{n}))^{2} + D^{2} (x_{1},\ldots,x_{n}) = 0$};
\path (0,2.2) node{$m := m + 1$};
\path (0,1.4) node{Print all tuples $(x_{1},\ldots,x_{n}) \in \mathbb{N}^{n}$ for which};
\path (0,1) node{max$(x_{1},\ldots,x_{n}) < m$ and $D(x_{1},\ldots,x_{n}) = 0$};
\path (0,.2) node{Stop};
\draw (0,6.2) ellipse (.75cm and 0.2cm);
\draw (-4.1,5.2) -- (-3.7,5.6) -- (4.1,5.6) -- (3.7,5.2) -- (-4.1,5.2);
\draw (-.55,4.4) rectangle (.55,4.8);
\draw (-2.9,2.8) rectangle (2.9,3.6);
\draw (-.8,2) rectangle (.8,2.4);
\draw (-4.2,.8) -- (-3.4,1.6) -- (3.8,1.6) -- (3.0,.8) -- (-4.2,.8);
\draw (0,.2) ellipse (.75cm and 0.2cm);
\draw (.8,2.2) -- (4,2.2) -- (4,4) -- (3.9,4); 
\draw[->] (0,2.8) -- (0,2.4) node[midway,right] {YES is returned};
\draw[->] (-2.5,2.8) -- (-2.5,1.6) node[above right] {NO is returned};
\draw[->] (0,6.0) -- (0,5.6);
\draw[->] (0,5.2) -- (0,4.8);
\draw[->] (0,4.4) -- (0,3.6);
\draw[->] (0,.8) -- (0,.4);
\draw[->] (3.9,4) -- (0,4);
\end{tikzpicture}\\[2ex]
\textbf{Flowchart~2:} An algorithm that conditionally finds all solutions to a Diophantine
equation which has less than $\kappa$ solutions in non-negative integers
\end{figure}
\noindent describes an algorithm whose execution terminates, if the set
\[
Sol(D):=\left\{(x_1,\ldots,x_n) \in {{\mathbb N}}^n\colon~D(x_1,\ldots,x_n)=0\right\}
\]
has less than $\kappa$ elements. If this condition holds, then the weakened
item~$(a)$ guarantees that the execution of Flowchart~2 prints all elements
of~{$Sol(D)$}. However, the weakened item~$(a)$ is equivalent to the original
one. Indeed, if the algorithm~${\cal A}$ satisfies the weakened item~$(a)$,
then Flowchart~3 illustrates a new algorithm~${\cal A}$ that satisfies the
original item~$(a)$.

\begin{figure}[H]
\centering
\begin{tikzpicture}[thick]
\ttfamily
\path (0,6.2) node{Start};
\path (0,5.4) node{Input a Diophantine equation $D$};
\path (0,4.6) node{$m := 1$};
\path (-1.2,3.4) node{Does $D$ have a solution in non-negative};
\path (-1.2,3) node{integers not greater than $m$?};
\path (-.75,2.2) node{Does the execution of $\mathcal{A}$};
\path (-.75,1.8) node{terminate after $m$ units of time?};
\path (3.75,2) node{$m := m + 1$};
\path (0,1) node{Print the output of $\mathcal{A}$};
\path (-3.75,.2) node{Print YES};
\path (0,.2) node{Stop};
\draw (0,6.2) ellipse (.75cm and 0.2cm);
\draw (-3.25,5.2) -- (-2.75,5.6) -- (3.25,5.6) -- (2.75,5.2) -- (-3.25,5.2);
\draw (-.55,4.4) rectangle (.55,4.8);
\draw (-4.75,2.8) rectangle (2.35,3.6);
\draw (-3.75,1.6) rectangle (2.25,2.4);
\draw (3,1.8) rectangle (4.5,2.2);
\draw (-2.5,.8) -- (-2,1.2) -- (2.5,1.2) -- (2,.8) -- (-2.5,.8);
\draw (-5,0) -- (-4.5,.4) -- (-2.5,.4) -- (-3,0) -- (-5,0);
\draw (0,.2) ellipse (.75cm and 0.2cm);
\draw (4.5,2) -- (5,2) -- (5,4) -- (4.65,4);
\draw[->] (4.75,4) -- (0,4);
\draw[->] (2.25,2) -- (3,2) node[midway,above] {No};
\draw[->] (0,6) -- (0,5.6);
\draw[->] (0,5.2) -- (0,4.8);
\draw[->] (0,4.4) -- (0,3.6);
\draw[->] (0,2.8) -- (0,2.4) node[midway,right] {No};
\draw[->] (0,1.6) -- (0,1.2) node[midway,right] {Yes};
\draw[->] (0,.8) -- (0,.4);
\draw[->] (-4.25,2.8) -- (-4.25,.4) node[above right] {Yes};
\draw[->] (-2.75,.2) -- (-.75,.2);
\end{tikzpicture}\\[2ex]
\textbf{Flowchart~3:} The weakened item $(a)$ implies the original one
\end{figure}
\par
Y. Matiyasevich in \cite{Matiyasevich3a} studies Diophantine equations
and Diophantine representations over ${\mathbb N} \setminus \{0\}$.
\begin{theorem}\label{the8} (\cite[p.~87]{Matiyasevich3a})
Suppose that there exists an effectively enumerable
set having no finite-fold Diophantine representation.
We claim that if a one-parameter Diophantine equation
\begin{equation}\label{eq4}
J(u,x_1,\ldots,x_m)=0
\end{equation}
for each value of the parameter $u$ has only finitely
many solutions in $x_1,\ldots,x_m$, then there exists
a number $n$ such that in every solution of~(\ref{eq4})
\[
x_1<u^n,\ldots,x_m<u^n
\]
\end{theorem}
\par
Theorem~\ref{the8} is false for $u=1$ when $J(u,x_1)=u+x_1-3$.
Theorem~\ref{the8} is missing in \cite{Matiyasevich3b}, the Springer edition of \cite{Matiyasevich3a}.
The author has no opinion on the validity of Theorem~\ref{the8} for integers $u>1$,
but is not convinced by the proof in \cite{Matiyasevich3a}.
Theorem~\ref{the8} restricted to integers $u>1$ and reformulated
for solutions in non-negative integers implies the following Corollary:\\[2ex]
{\bf Corollary.} {\em If there exists a recursively enumerable set having no
finite-fold Diophantine representation, then any set ${\cal M} \subseteq {\mathbb N}$
with a finite-fold Diophantine representation is computable.}
\newpage
Let us pose the following two questions:
\begin{question}\label{que1}
Is there an algorithm ${\cal B}$ which takes as input a Diophantine equation
$D$, returns an integer, and this integer is greater than the heights of
non-negative integer solutions, if the solution set has less than $\kappa$
elements? We allow a possibility that the execution of {${\cal B}$} does not
terminate or ${\cal B}$ returns nothing or something irrelevant, if $D$ has at
least $\kappa$ solutions in non-negative integers.
\end{question}

\begin{question}\label{que2}
Is there an algorithm ${\cal C}$ which takes as input a Diophantine equation
$D$, returns an integer, and this integer is greater than the number of
non-negative integer solutions, if the solution set is finite? We allow a
possibility that the execution of ${\cal C}$ does not terminate or ${\cal C}$
returns nothing or something irrelevant, if $D$ has infinitely many solutions
in non-negative integers.
\end{question}

Obviously, a positive answer to Question~\ref{que1} implies the weakened
item~$(a)$. Conversely, the weakened item~$(a)$ implies that Flowchart 4
describes an appropriate algorithm~${\cal B}$.

\begin{figure}[H]
\centering
\begin{tikzpicture}[thick]
\ttfamily
\path (0,5.0) node{Start};
\path (0,4.2) node{Input a Diophantine equation $D (x_{1},\ldots,x_{n}) = 0$};
\path (0,3.4) node{$m := 0$};
\path (0,2.2) node{Execute $\mathcal{A}$ on the equation};
\path (0,1.8) node{$(m + y - (x_{1} + \ldots + x_{n}))^{2} + D^{2} (x_{1},\ldots,x_{n}) = 0$};
\path (0,1) node{$m := m + 1$};
\path (-2.5,.2) node{Print $m$};
\path (0,.2) node{Stop};
\draw (0,5.0) ellipse (.75cm and 0.2cm);
\draw (-4,4) -- (-3.6,4.4) -- (4,4.4) -- (3.6,4) -- (-4,4);
\draw (-.55,3.2) rectangle (.55,3.6);
\draw (-2.9,1.6) rectangle (2.9,2.4);
\draw (-.8,.8) rectangle (.8,1.2);
\draw (-3.65,0) -- (-3.25,.4) -- (-1.25,.4) -- (-1.65,0) -- (-3.65,0);
\draw (0,.2) ellipse (.75cm and 0.2cm);
\draw (.8,1) -- (4,1) -- (4,2.8) -- (3.9,2.8);
\draw[->] (0,1.6) -- (0,1.2) node[midway,right] {YES is returned};
\draw[->] (-2.5,1.6) -- (-2.5,.4) node[above right] {NO is returned};
\draw[->] (0,4.8) -- (0,4.4);
\draw[->] (0,4) -- (0,3.6);
\draw[->] (0,3.2) -- (0,2.4);
\draw[->] (0,1.6) -- (0,1.2);
\draw[->] (4,2.8) -- (0,2.8);
\draw[->] (-1.45,.2) -- (-.75,.2);
\end{tikzpicture}\\[2ex]
\textbf{Flowchart~4:} The weakened item $(a)$ implies a positive answer to Question~\ref{que1}
\end{figure}
\begin{theorem}\label{the9}
(\cite{Tyszka4}) A positive answer to Question~\ref{que1} for $\kappa=\omega$
is equivalent to a positive answer to Question~\ref{que2}.
\end{theorem}
\begin{proof}
Trivially, a positive answer to Question~\ref{que1} for $\kappa=\omega$
implies a positive answer to Question~\ref{que2}.
Conversely, if a Diophantine equation {$D(x_1,\ldots,x_n)=0$} has only
finitely many solutions in non-negative integers, then the number of
non-negative integer solutions to the equation
\[
D^2\left(x_1,\ldots,x_n\right)+\left(x_1+\ldots+x_n-y-z\right)^2=0
\]
is finite and greater than $\textrm{max}(a_1,\ldots,a_n)$, where
$(a_1,\ldots,a_n) \in {{\mathbb N}}^n$ is any solution to
$D(x_1,\ldots,x_n)=0$.
\begin{flushright}
$\qed$
\end{flushright}
\end{proof}
\bibliography{references_expanded}
\end{document}